\documentclass{article} 
\usepackage{amsmath,amsfonts,amsbsy,amsthm,latexsym,amssymb,enumerate} 
\usepackage{graphicx}
\usepackage{graphics} 
\usepackage{float}
\usepackage{pictexwd} 
\usepackage{tabularx}

\newtheorem{theorem}{Theorem}[section] 
 
\newtheorem{lemma}[theorem]{Lemma} 
 
\newtheorem{definition}[theorem]{Definition}

\begin{document} 
\sloppy 
\title{~\\[-6ex] Simulations and a conditional limit theorem for intermediately subcritical branching processes in random environment }
\author{ \textsc{Ch. B\"oinghoff} 
  \thanks{Fachbereich Mathematik, Universit\"at Frankfurt, Fach 
     187, D-60054 Frankfurt am Main, Germany , boeinghoff@math.uni-frankfurt.de, kersting@math.uni-frankfurt.de} \hspace{.8cm} 
   \textsc{G. Kersting} $^*$}
 \maketitle 
\begin{abstract} 
Intermediately subcritical branching processes in random environment are at the borderline between two subcritical regimes and 
exhibit a particularly rich behavior. In this paper, we prove a functional limit theorem for these processes. It is discussed together with two other recently proved limit theorems for the intermediately
 subcritical case and illustrated by several computer simulations.     
\end{abstract}    

\begin{small} 
\emph{Running head.}  Branching processes in random environment 
 \end{small} 
\thispagestyle{empty} 

\newpage

\section{Introduction and main results} 
Branching processes in random environment are a discrete time model for the development of a population of individuals. In contrast to Galton-Watson processes, it is 
assumed that individuals are exposed to a random environment, which influences the reproductive success of the individuals and varies from one generation to the next in an 
i.i.d. manner. Given the environment, all individuals reproduce independently according to the same mechanism. 

More precisely, let $Q$ be a random variable taking values in the space $\Delta$ of probability measures on $\mathbb{N}_0$. Equipped with the total variation metric, 
$\Delta$ is a Polish space. An infinite sequence $\Pi=(Q_1,Q_2,\ldots)$ of i.i.d. copies of $Q$ is called \textsl{random environment}. We denote by $Q_n$ the offspring 
distribution of an individual in generation $n-1$ and by $Z_n$ the number of individuals in generation $n$. A sequence of $\mathbb{N}_0$-valued random 
    variables      $Z_{0},Z_{1},\ldots$ is called a 
 \emph{branching process in the random environment (BPRE)} $\Pi$, if $Z_{0} 
  $ is      independent of $ \Pi$   and   given 
   $ \Pi$ the  process $Z=(Z_{0},Z_{1},\ldots)$ is a Markov chain 
with 
 \begin{equation}  \label{transition} 
    \mathcal{L} \big(Z_{n} \; \big| \; Z_{n-1}=z, \, \Pi 
   = 
     (q_{1},q_{2},\ldots) \big) \ = \ q_{n}^{*z} 
\end{equation} 
  for every $n\in \mathbb{N}$, $z \in \mathbb{N}_0$ and 
     $q_{1},q_{2},\ldots \in \Delta$, where 
$q^{*z}$  is the $z$-fold 
convolution of the measure $q$. The corresponding probability  measure on the underlying probability space will be denoted by $\mathbb{P}$. 
In the following 
we assume  that the process starts with a single founding 
ancestor, $Z_0=1$ a.s. Throughout the paper, we shorten  $Q(\{y\}), q(\{y\})$ to $Q(y),q(y)$.

As it turns out, the fine structure of the offspring distributions is of secondary importance for the asymptotics of the BPRE. More precisely, 
the asymptotics of $Z$ is mainly determined by the \textsl{associated random walk}, which only contains information on the mean offspring number 
in each generation. 
The associated 
  random walk $S=(S_n)_{n\ge 0}$ is the random walk having initial state
$S_{0}=0$ and  increments   $X_n=S_n-S_{n-1}, \, n\ge 1$ 
defined by
\[ 
    X_{n} \ = \ \log m(Q_n),  
\] 
where 
\[ 
m(q)\ = \
\sum_{y=0}^{\infty} y  q(y) 
\] 
is the mean of the offspring distribution $q \in \Delta$. 
Assuming $Z_0=1$ a.s., it results that  
      the conditional expectation of 
 $Z_{n}$ given the environment~$ \Pi$ 
can be written as
\begin{equation}  \label{expect} 
  \mathbb{E}[Z_{n} \,| \,  \Pi \,]   \ = \ 
\prod_{k=1}^n m(Q_k)\ = \ e^{S_n} \quad 
\mathbb{P}\text{--a.s.} 
\end{equation} 
Averaging over the environment yields
\begin{equation}  \label{expect2} 
  \mathbb{E}[Z_{n} ]   \ = \ 
 \mathbb{E}[ e^{S_n} ] \ = \ \mathbb E[e^X]^n 
\end{equation}
with $X= \log m(Q)$.

There are several phase transitions present in the class of BPRE. They already become visible, when looking at the asymptotic survival probability.  From Jensen's inequality for $0 \le \lambda \le 1$
\[ \mathbb P(Z_n>0) \le \mathbb E[Z_n^\lambda] =\mathbb E\big[\mathbb E[Z_n^\lambda\mid \Pi]\big]\le \mathbb E\big[\mathbb E[Z_n\mid \Pi]^\lambda\big] = \mathbb E[e^{\lambda X}]^n \ . \]
Indeed under quite general conditions (see \cite{af_80,de,gl,gkv}) it holds that
\[ \lim_{n \to \infty} \mathbb P(Z_n>0)^{\frac 1n} = \inf_{0\le \lambda \le 1} \mathbb E[e^{\lambda X}] \ . \]

The formula suggests where the  phase transitions are located. This depends on the value $\lambda_{\min}$, where the moment generating function $\lambda \mapsto \mathbb E[e^{\lambda X}]$ has its minimum. One phase transition appears, when $\lambda_{\min} =0$, which essentially amounts to the condition $\mathbb E[X]=0$. Then $S$ is a recurrent random walk, and $Z$ is called a {\em critical} BPRE. For a detailed study we refer to \cite{agkv}.

The other phase transition turns up when $\lambda_{\min} =1$, which matches to the condition \[\mathbb E[Xe^X]=0\ . \] 
This condition in turn implies $\mathbb E[e^X]<1$ and $\mathbb E[X]<0$. Then $Z$ is called an {\em intermediately subcritical} BPRE and $S$ is a transient random walk with negative drift.

In the other cases $\lambda_{\min}<0$, $0<\lambda_{\min}<1$ and $\lambda_{\min}>1$ the BPRE $Z$ is called {\em supercritical}, {\em weakly subcritical} and {\em strongly subcritical}, respectively, a classification, which goes back to Afanasyev \cite{af_80} and Dekking \cite{de}.

As one would expect, BPREs exhibit a particularly rich behavior in the two transition cases. Here we focus on the intermediately subcritical case, namely on the behavior of the process up to time $n$, 
conditioned on the event $\{Z_n>0\}$, in the limit $n \to \infty$. A main concern of our paper is to exemplify its features by means of computer simulations. In doing so we shall discuss three functional limit theorems, which underly these simulations. Two are taken from the publication \cite{abkv11} (being the main basis of the present paper). The other limit theorem is proved below, which makes the second part of the paper. Intermediately subcritical BPREs have also been studied in \cite{af_01,gkv,va_04}. For a comparative discussion we refer the reader to \cite{bgk}.

For an intermediately subcritical BPRE it is natural to introduce a change to the probability measure $\mathbf{P}$, given by its expectation
\[ \mathbf{E}[ \varphi(Q_1,\ldots,Q_n, Z_0,\ldots,Z_n)] \ = \ \gamma^{-n} \mathbb{E}\big[\varphi(Q_1,\ldots,Q_n, Z_0,\ldots,Z_n)e^{S_n}\big]  \] 
for any $n\in\mathbb{N}$ and 
any measurable, bounded function $\varphi:\Delta^n\times \mathbb{N}_0^{n+1}\rightarrow\mathbb{R}$, 
with $ \gamma^n= \mathbb E[e^{S_n}]= \mathbb E[Z_n]$, thus
\[ \gamma \ = \ \mathbb{E} [e^{X}] \ . \]  
The condition $ \mathbb{E}[Xe^{X}]=0$ translates to
\[ \mathbf{E}[X] \ = \ 0 \ . \]
Therefore $S$ becomes a recurrent random walk under $\mathbf{P}$.

Let us pass to the assumptions. For $a \in \mathbb N$ denote
\[ \zeta(a) \ = \ \sum_{y=a}^\infty y^2 Q(y) \Big/ m(Q)^2 \ , \quad a \in \mathbb{N} \ , \]
which we refer to as the standardized truncated second moment of $Q$.

\paragraph{Assumption A.} {\em Let $X$ be non-lattice with 
\[  \mathbf{E} [ X] \ = \ 0  \ , \   0<\mathbf{E} [ X^2]<\infty \ .   \] 
Moreover let
\[  \mathbf{E} \big[ (\log^+ \zeta(a))^{2 + \varepsilon}\big] \ < \ \infty \]
for some $a \in \mathbb N$ and $\varepsilon >0$, where} $\log^+ x= \log(x \vee 1)$.\\

\noindent
The condition $\mathbf{E} [ X^2]<\infty $ is taken for convenience here. In the second half of the paper we shall replace it by a weaker assumption. See \cite{agkv} for examples where the last assumption is fulfilled. In particular our result holds for binary branching processes in random environment 
(where individuals have either two children or none) and for cases where $Q$ is a.s. a Poisson distribution or a.s. a geometric distribution.

Our first functional limit theorem concerns the stochastic processes $S^n=(S^n_t)_{0\leq t\leq 1}$, $n \in \mathbb N$, given by
\[S_t^n= S_{nt} \ , \quad 0\le t \le 1 \ .\]
Here and in the sequel we always shorten $\lfloor nt \rfloor$ to $nt$. Donsker's theorem  states that
\[ \frac{S^n}{\sigma\sqrt{n}}  \stackrel{d}{\to} W\]
in distribution on the Skorohod space $D[0,1]$, where $W=(W_t)_{0\leq t\leq 1} $ is a standard Brownian motion and
\[ \sigma^2 = \mathbf E[X^2] \ . \]
We denote by
\[ W^c =(W_t^c)_{0 \le t \le 1} \]
a process, which we call here a {\em conditional Brownian motion}, that is a Brownian motion conditioned to take its minimal value at $t=1$.

\begin{theorem}\label{theo1} 
Under Assumption A, it holds that as $n\rightarrow\infty$,
\[ \Big(\frac{S^n}{\sigma\sqrt{n}} \ \Big|\ Z_n>0\Big) \stackrel{d}{\to} W^c\]
in the Skorohod space.
\end{theorem}
\noindent
This theorem is taken from \cite{abkv11}[Theorem 1.3]. The statement turns out to be characteristic for intermediately subcritical BPREs.
It describes the impact on the random environment resulting from  conditioning on the event $\{Z_n>0\}$. Since $\mathbb P(Z_n>0)\le \mathbb E[Z_n]=\mathbb E[e^X]^n$,  note that
$\mathbb P(Z_n>0)$ is exponentially small such that we are in the range of large deviations. As usual there are distinct scenarios leading to different exponential 
rates for the probabilities, i.e. they require  different 'costs'. Let us discuss this trade-off in detail. 

First it is to be expected that $Z_n$ is asymptotically of order $O_P(1)$ conditioned on the event $\{Z_n >0\}$, since it would be too costly to build up a larger population. 
This enforces that $S_0, \ldots,S_n$ have their minimum close to the end, because otherwise the population would have the chance for a late growth. Theorem \ref{theo1}  confirmes this consideration, 
yet the same phenomenon turns up also for the weakly and strongly subcritical case, see \cite{abkv,agkv1}.

Next note that among random walk paths $S_0, \ldots,S_n$ with a late minimum  one can imagine two strategies for the BPRE $Z$ to survive until time $n$. Either $S$ builds up one big upward excursion. This is difficult to realize for a random walk with negative drift, but it provides an environment
 in which the branching process easily survives. Or $S$ is on and on decreasing. This is readily realized for a random walk with negative drift, but gives the branching process  a hard time to survive. Now the first alternative is realized for weakly subcritical  and the second for strongly subcritical BPREs, see \cite{abkv,agkv1}. 
Theorem \ref{theo1} indicates that in the intermediately subcritical case both possibilities compete with each other, which means that they are equally costly. Indeed upward excursions alternate with decreasing ladder points within $W^c$.

Theorem \ref{theo1} is our basis for the computer simulations of the conditional BPRE, given $\{Z_n>0\}$. Since this event has exponentially small probability, direct simulations are not realizable. Also an access via a suitable Doob $h$-transform seems out of reach. Therefore we present an approximate solution by first 
simulating a conditional random walk path $S_0^c,\ldots,S_n^c$  that is a random walk path $S_0, \ldots,S_n$ conditioned to have its minimum at time $n$. 
Here we also rely on additional information supplied by \cite{abkv11}[Theorem 1.3]: If   $\tau_n$ denotes the moment, when $S_0, \ldots,S_n$ attains 
its minimum for the first time, then   $n-\tau_n$ given $Z_n>0$ is convergent in distribution for $n \to \infty$. Given $S_0^c,\ldots,S_n^c$ we generate the random environment. We choose a situation where the random walk completely determines 
the environment (otherwise we could not simply rely on  Theorem \ref{theo1}). From a computational point of view it is convenient to choose geometric offspring distributions. 
Then given the environment the branching process $1=Z_0,Z_1, \ldots,Z_n$ conditioned to survive is efficiently generated by a general construction of the 
conditioned BPRE due to Geiger \cite{ge_99}. Altogether we replace the annealed situation in a way by a related quenched setting, which admittedly is only a substitute. 
The details are presented in Section \ref{sec2}. 

\paragraph{Remark.} The asymptotic shape of the limit distribution of $n-\tau_n$ given the event $\{Z_n>0\}$ can be easily derived: For $0\le k \le n$
\begin{align*}
\mathbb P(&n-\tau_n=k \mid Z_n >0) \\ &\le \frac{\mathbb P(\tau_{n-k}=n-k,Z_{n-k}>0) \mathbb P(\min(S_1, \ldots,S_k) \ge 0)}{\mathbb P(Z_n>0)} \ . 
\end{align*}
From \cite{abkv11} $\mathbb P(Z_{n-k}>0) \sim \gamma^k \mathbb P(Z_n>0)$ and $\mathbb P(\tau_n=n \mid Z_n>0)$ has a positive limit. From \cite{abkv}[Proposition 2.1]  
\[ \mathbb P(\min(S_1, \ldots,S_k) \ge 0) = \gamma^{k} \mathbf E[ e^{-S_k};\min(S_1, \ldots,S_k) \ge 0] \sim c'\gamma^k k^{-3/2}  \]
with some $c'>0$. Thus there is a $c>0$ such that for all $k \ge 0$
\[ \lim_{n \to \infty} \mathbb P(n-\tau_n=k \mid Z_n >0) \le \frac c{k^{3/2}} \ . \]
With some effort this upper estimate can be refined to an asymptotic  equality.\\\\
The following picture shows a path $S_0^c,\ldots,S_n^c$ of length $n=1000$.
\begin{center}
\includegraphics[angle=0,width=0.75\textwidth]{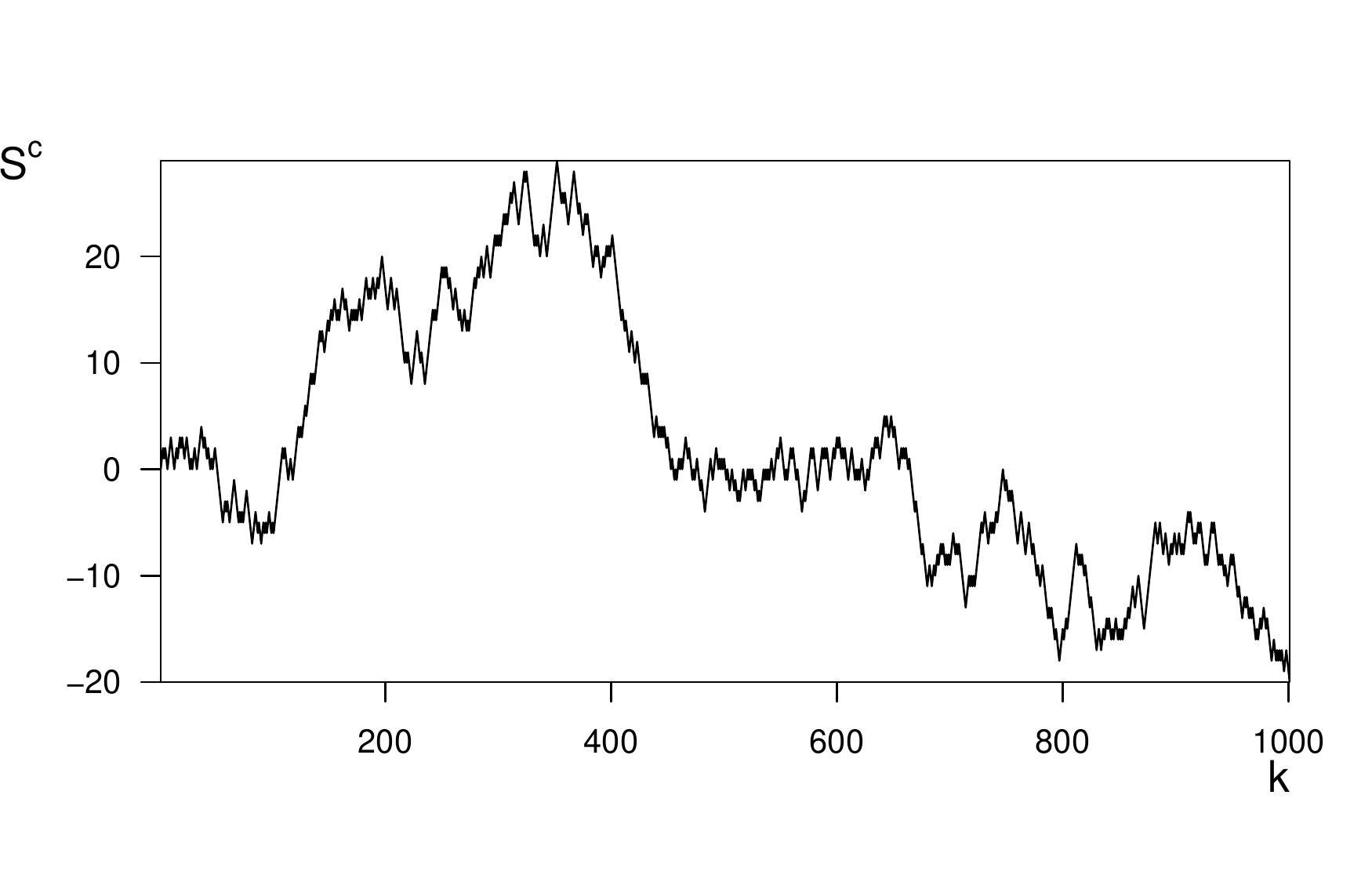}
\end{center}

From Theorem \ref{theo1} we expect the following behavior of $Z_0,Z_1, \ldots,Z_n$, conditioned on $Z_n>0$: At decreasing ladder points of $S$ the size $Z_k$ is close (or equal) to 1, whereas during upward excursions of $S$ the population size $Z_k$ is tied to the up and down of its 
conditional expectation $\mathbb E[ Z_k \mid \Pi]= e^{S_k}$ a.s. More precisely one would suspect that in either case $\log Z_k$ is close to 
$S_k-\min_{j\le k} S_j$, which is the height of the random walk relative to its height at the last ladder point. This leads us to conjecture that as $n\rightarrow\infty$
\[  \Big(\frac 1{\sigma \sqrt n} \log Z_{nt}  \ \Big|\ Z_n >0\Big) \stackrel{d}{\to}  W^r_t   \]
for $0 \le t \le 1$, where
\[  W^r_t=W_t^c - \min_{s \le t} W_s^c \ , \quad 0 \le t \le 1 \ ,\]
is the conditional Brownian motion {\em reflected} at its current minimum. Now a famous result of L\'evy says that for the unconditional Brownian motion $W$ the processes $W'=(W_t - \min_{s \le t} W_s)_{0\le t \le 1}$  and $(|B_t|)_{0 \le t \le 1}$ are equal in distribution, where $B$ denotes another Brownian motion. Conditioning $W$ to take its minimum at time 1 is equivalent to conditioning $W'_1$ to take the value 0. For $B$ this means that we pass over to a Brownian bridge.

\begin{theorem}\label{theo2} Under Assumption A, it holds that as $n\rightarrow\infty$,
\[ \Big( \Big(\frac 1{\sigma \sqrt n} \log Z_{nt}\Big)_{0\le t \le 1} \ \Big|\ Z_n >0\Big) \stackrel{d}{\to} |B|  \]
in the Skorohod space, where $B=(B_t)_{0\le t \le 1}$ denotes a standard Brownian bridge.
\end{theorem}

\noindent
For the linear fractional case convergence of finite dimensional distributions was obtained by Afanasyev \cite{af_01} (without identifying the limit). This theorem is proved in Section \ref{th}. We will now illustrate it by some simulations. The next  figure shows two paths corresponding to the path $S_0^c,\ldots,S_n^c$ of Figure 1: In black the path $S_0^r, \ldots, S_n^r$, given by
\[ S^r_k = S^c_k - \min_{j \le k} S^c_j \ ,\]
that is the conditional random walk, reflected at its current minimum, and in grey the path $\log Z_0, \ldots, \log Z_n$, where $1=Z_0, \ldots,Z_n$ denotes the branching process, given the environment determined by $S_0^c, \ldots, S_n^c$ and conditioned to survive within this environment. The fit of both paths is clearly visible.
\begin{center}
\includegraphics[angle=0,width=0.75\textwidth]{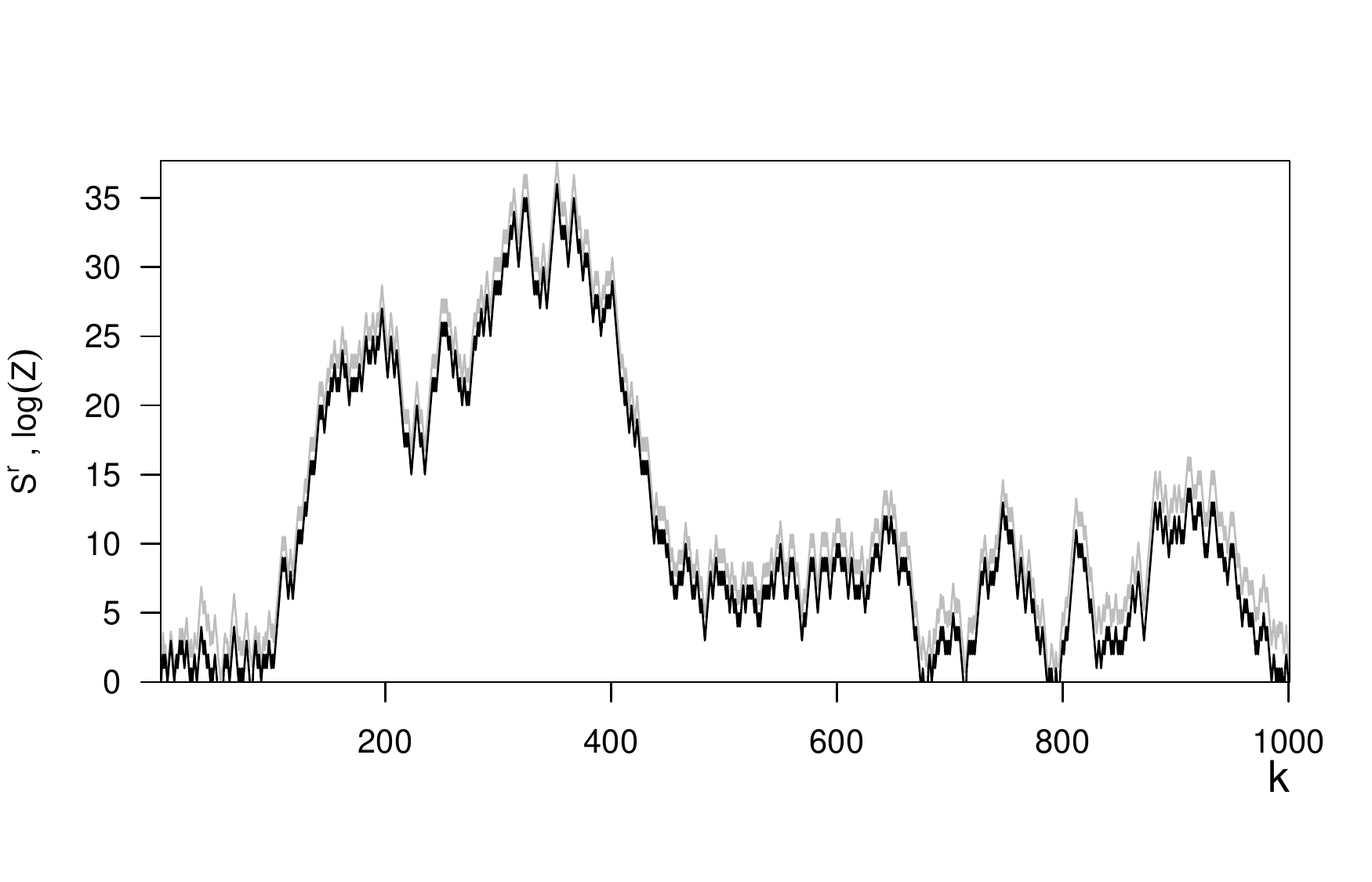}
\end{center}

The next theorem focuses on the difference between the grey and black processes in the last figure as well as on the magnitude of the population in ladder points of the random walk. It is taken from \cite{abkv11}[Theorem 1.4]. 
Recall that $\tau_{nt}$ is the moment, when $S_0, \ldots, S_{nt}$ takes its minimum.

\newpage

\begin{theorem}\label{theo3} Let $0 < t_1 < \cdots < t_r < 1$. For $i=1, \ldots,r$ let
\[ \mu(i) = \min\big\{ j \le i: \inf_{t \le t_j}W^c_t = \inf_{t \le t_i}W^c_t\big\} \ .\]
Then under Assumption A there are i.i.d. random variables $U_1, \ldots,U_r$ with values in $\mathbb N$  and independent of $W^c$ such that
\begin{align*}
  \big((Z_{ \tau_{nt_1}},\ldots,Z_{\tau_{nt_r}} ) \mid Z_n>0\big) 
 \stackrel{d}{\rightarrow}  (U_{ \mu(1)},\ldots,U_{ \mu(r)})  
\end{align*}
as $n \to \infty$. Also there are i.i.d. strictly positive random variables $V_1,  \ldots,V_r$ independent of $W^c$ such that
\begin{align*}
  \Big(\big( \frac{Z_{  n t_1 }}{e^{S_{  n t_1 }-S_{\tau_{nt_1}}} }  ,\ldots, \frac{Z_{  nt_r }}{e^{S_{  n t_r }-S_{\tau_{nt_r}}}}  \big) \ \big|\ Z_n>0\Big)   \stackrel{d}{\to}  (V_{ \mu(1)},\ldots,V_{ \mu(r)})   
\end{align*}
as $n \to \infty$.
\end{theorem}

\noindent
The first part confirms that the population size is of order $O(1)$ in ladder points. The meaning of the second statement has already been explained in \cite{abkv11}. Shortly speaking: Within upward excursions $Z_k/\exp(S_k- \min_{j \le k} S_j)$ takes asymptotically a constant value, whereas this value changes in an independent manner from one excursion to the next. This is expressed in the next two pictures. The first shows $Z_k/ \exp(S^r_k)$ for the random path $S_0^c, \ldots, S_n^c$ from Figure 1. 

\begin{center}
\includegraphics[angle=0,width=0.75\textwidth]{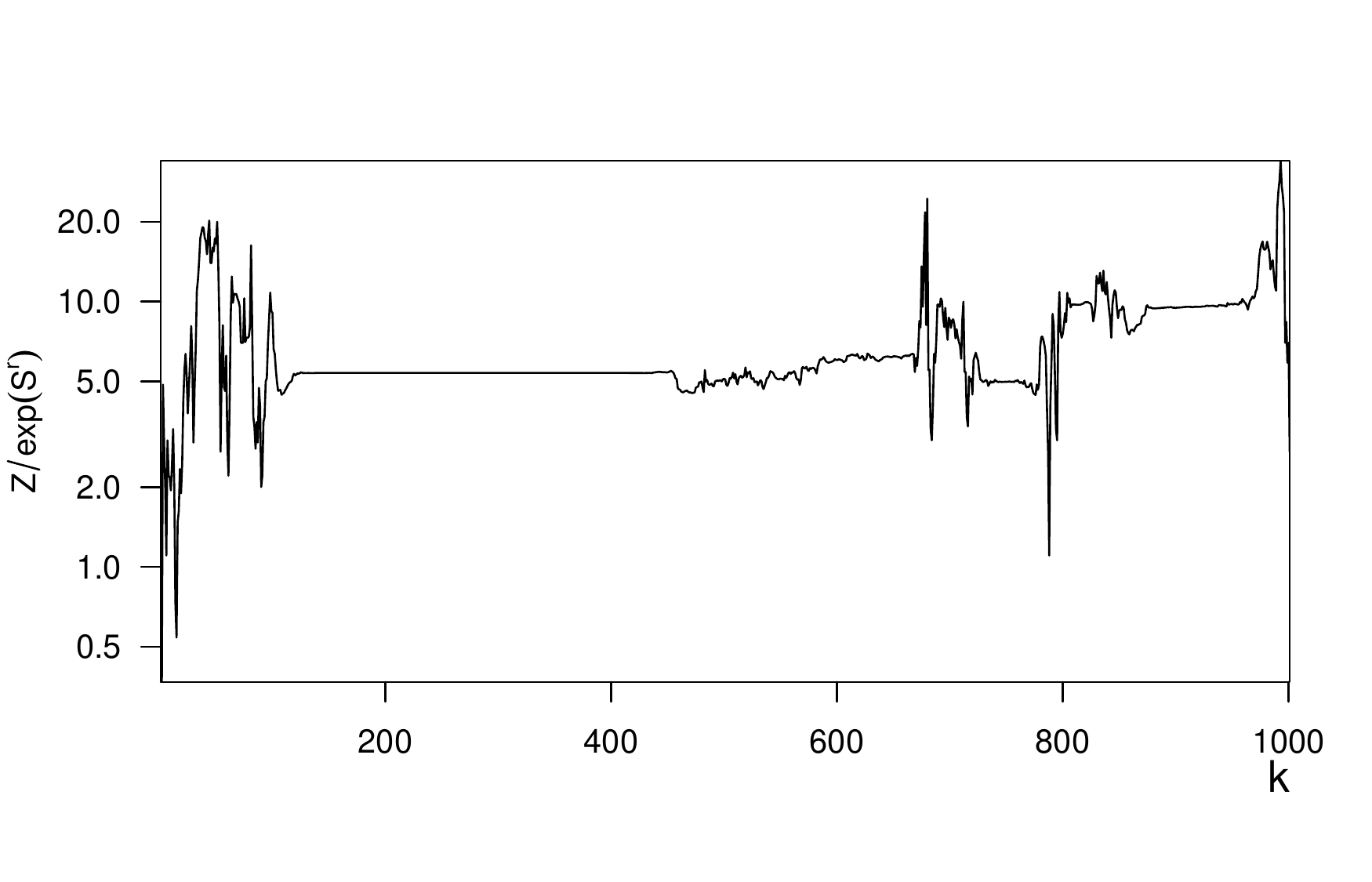}
\end{center}

Note that Theorem \ref{theo3} only deals with finite dimensional distributions, it cannot be extended to a functional limit theorem in Skorohod space in a standard manner, since the limiting process would consist of paths being constant within Brownian excursions but independent between different excursions. This leads to paths which are not c\`adl\`ag. This becomes manifest in the next picture with $n=100 000$. Note the heavy oscillations of the path between its constant sections.

\begin{center}
\includegraphics[angle=0,width=0.75\textwidth]{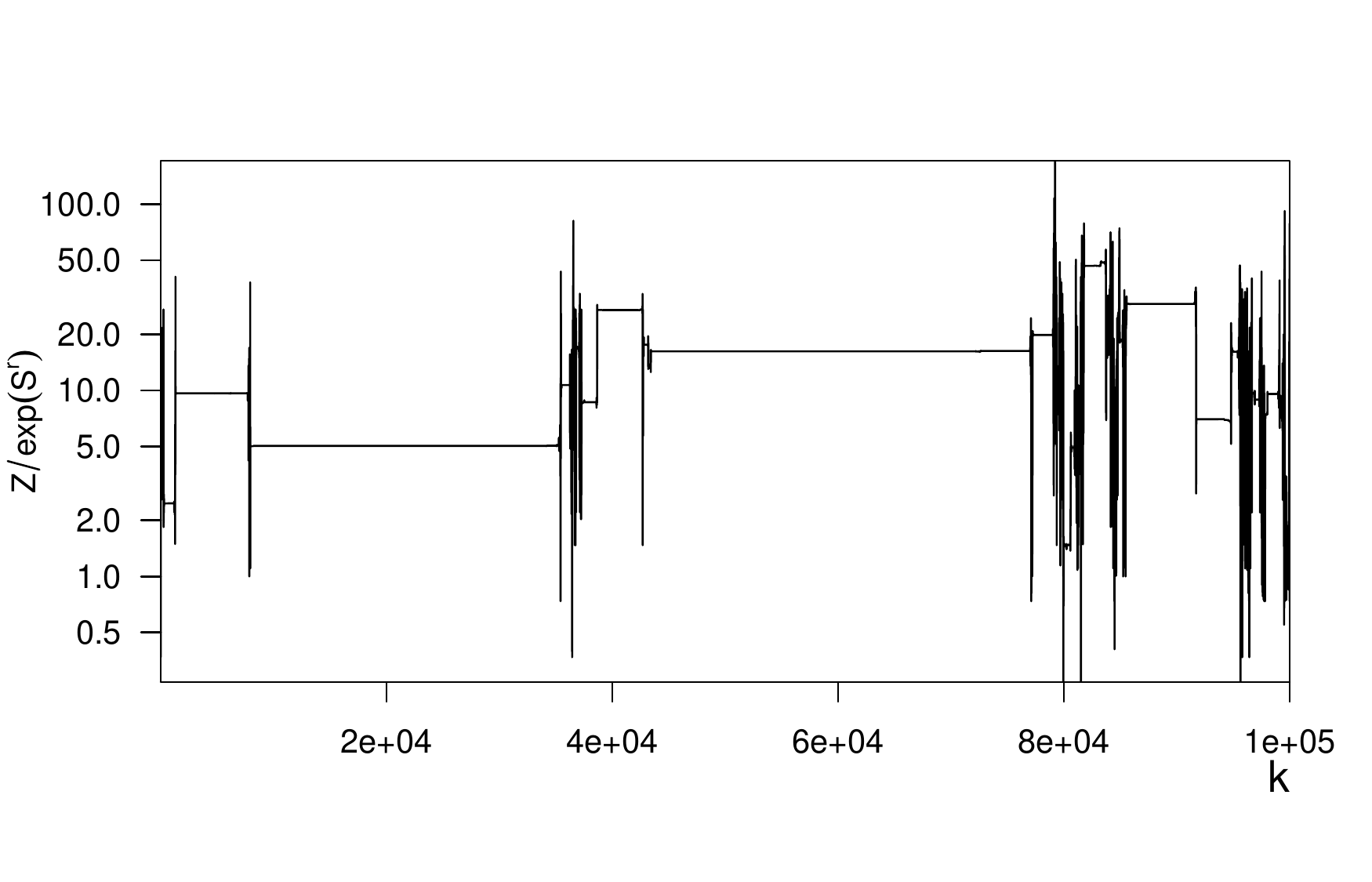}
\end{center}

The final picture shows the magnitude of the population, when restricted just to the strictly decreasing  ladder points of $S_0^c,\ldots, S_n^c$, which are 49 in our $n=1000$ example. One observes that this process is not just white noise. The dependence structure does not seem to be easily captured.

\begin{center}
\includegraphics[angle=0,width=0.75\textwidth]{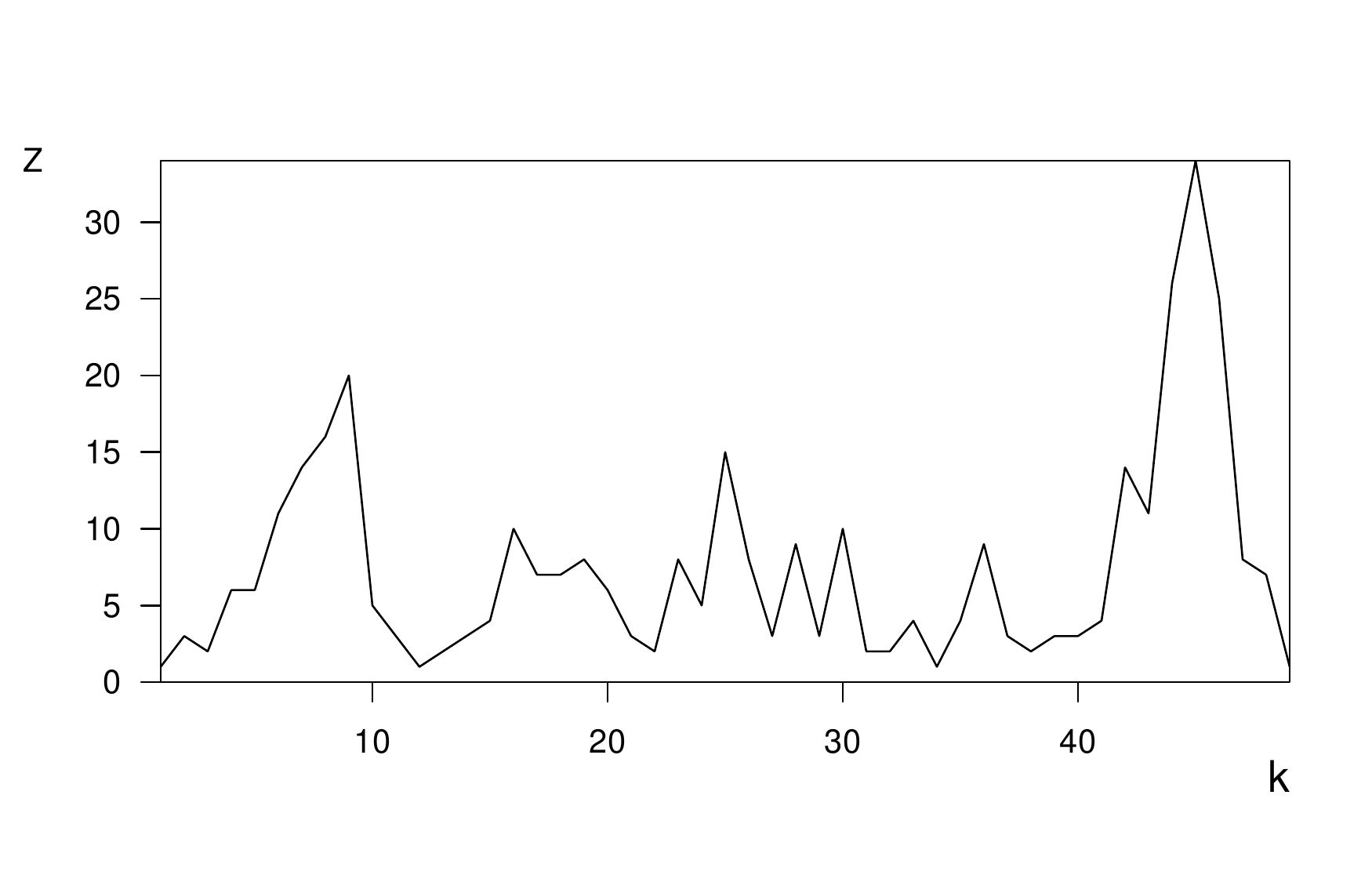}
\end{center}

The rest of the paper is organized as follows. In Section \ref{sec2} we assemble the facts, relevant for the simulations. In Section \ref{th} we give the proof of Theorem \ref{theo2} in a more general setting.

\section{Simulation of branching processes in random environment}\label{sec2}
In this section, we derive and describe the simulation algorithm which was used to sample the simulations presented in the previous section. 
\subsection{Simulation of a conditional random walk}
Theorem \ref{theo1} says that asymptotically as $n\rightarrow\infty$, the associated random walk is distributed like a random walk conditioned on having its minimum at the end. 
Thus here it is our concern to sample paths $S^c_0,S_1^c, \ldots, S_n^c$. Let us introduce for $n \ge 1$
\[  M_n \ = \ \max(S_1, \ldots,S_n) \ .\]
Introducing the dual random walk
\[\hat{S}_{n-i} =S_n-S_i \ , \ 0\leq i\leq n\ ,\]
we get that
\begin{align*}
\big((S_0, \ldots,S_i,\ldots,& S_n) \mid S_n < \min(S_0, \ldots, S_{n-1})\big)\\ &\stackrel d=  \big(\hat S_0, \ldots, \hat S_n- \hat S_{n-i}, \ldots, \hat S_n)  \mid \max(\hat S_1, \ldots, \hat S_n)<0\big) \ .
\end{align*}
Thus a random walk path, conditioned on $\{S_n < \min(S_0, \ldots, S_{n-1})\}$ can be sampled by simulating a random walk path, conditioned on $\{M_n<0\}$
and then deriving therefrom the dual path. Here, we only consider a simple random walk, i.e.
\[\mathbf{P}(X=1)=\mathbf{P}(X=-1)=\frac{1}{2} \ .\]
By Markov property, the distribution of $X_k$ conditioned on $\{M_n<0\}$ is for $1\leq k\leq n$ given by
\begin{align}
 \mathbf{P} (X_{k}=1\mid M_n<0, S_{k-1}=x) &= \frac 12 \frac{\mathbf{P}(M_n<0\mid S_{k}=x+1)}{\mathbf{P}(M_n<0\mid S_{k-1}=x)}\nonumber\\
&= \frac 12 \frac{\mathbf{P}(M_{n-k}<-x-1)}{\mathbf{P}(M_{n-k+1}<-x)} \ .\label{prw}
\end{align}
By the reflection principle, the distribution of the maximum is given by
\[\mathbf{P}(M_n\geq y)=\mathbf{P}(S_n\geq y)+\mathbf{P}(S_n>y) \ .\] 
For a simple random walk
\[\mathbf{P}(S_n=y)=\mathbf{P}\big(Y= (n-y)/2\big) \ ,\]
where $Y$ is binomially distributed with parameters $(n,1/2)$. Thus the probability in (\ref{prw}) is easily calculated and sampling paths 
of the conditional random walk is straightforward.

\subsection{Geiger construction}

In this section, we treat branching processes in varying environment, i.e. the 
environment $\pi=(q_1,q_2,\ldots)$ is fixed. We denote the underlying probability measure by $\mathbf P_\pi(\cdot)=\mathbb{P}(\cdot \mid \Pi=\pi)$. 

The Geiger construction allows to construct a branching process in varying environment, conditioned on $\{Z_n>0\}$, along its ancestral line 
(see e.g. \cite{ge_99, abkv11}). Following the notation in \cite{abkv11}, we denote by $\mathcal T_n$ 
the set of all ordered rooted trees of height exactly $n$, $n\in\mathbb{N}_0$ and 
$$\mathcal T_{\ge n}=\mathcal T_n\cup \mathcal T_{n+1} \cup \cdots \cup \mathcal T_\infty$$ the set of all trees of height at least $n$.

Let us introduce the operation $[\ ]: \ \mathcal T_{\ge n}\to \mathcal T_n$ of pruning a tree of height $\geq n$ to a tree of height exactly $n$ by eliminating 
all nodes of larger height.

A tree with a stem (following \cite{abkv11} called a {\em trest}) is defined by a pair
 \[ \mathsf{t} = (t, k_0 k_1\ldots k_n) \ , \]
where $t \in \mathcal T_{\ge n}$ and $k_0, \ldots,k_n$ are nodes in $t$ such that $k_0$ is the root 
(founding ancestor) and $k_{i}$ is an offspring of $k_{i-1}$.

The operation
\[ \langle t \rangle_n = ([t]_n, k_0(t) \ldots k_n(t))  \]
maps an ordered and rooted tree of height at least $n$ to the associated, unique trest of height $n$, 
where $k_0(t) \ldots k_n(t)$ is the {\em leftmost} stem, which can be fitted into $[t]_n$. 

Now, we are able to construct the conditional branching tree along its ancestral line. Let 
\[ \mathsf T_{n,\pi}= (T_n, K_0 \ldots K_n) \]
denote a random trest of height $n$ and let for $i=1, \ldots, n$
\begin{align*} 
T_i' &= \text{subtree within } T_n \text{ right to the stem with root } K_{i-1} \\ 
T_i'' &= \text{subtree within } T_n \text{ left to the stem with root } K_{i-1}  \\
R_i &= \text{size of the first generation of } T_i'   \\
L_i &= \text{size of the first generation of } T_i''   
\end{align*}
\begin{definition}[Geiger tree]
A random trest $\mathsf T_{n,\pi}$ is called a Geiger tree in the environment $\pi$ and its distribution is uniquely determined (up to possible offspring of $K_n$)
 if the following properties are fulfilled:
\begin{itemize}
\item The joint distribution of $R_i$ and $L_i$ is given by 
\begin{align}
 \mathbf P_\pi& ( R_i=r, L_i=l) \nonumber \\
&= q_i(r+l+1) \frac{\mathbf P_\pi(Z_n >0 \mid Z_i=1) \mathbf P_\pi(Z_n=0 \mid Z_i=1 )^{l}}{\mathbf P_\pi(Z_n >0 \mid Z_{i-1}=1)}\label{geigerdis}
\end{align}
\item $T_i'$ decomposed at its first generation consists of $R_i$ subtrees $\tau_{ij}'$, $j=1, \ldots, R_i$, which are branching trees within the fixed environment 
$(q_{i+1},q_{i+2}, \ldots)$.
\item  $T_i''$ consists of $L_i$ subtrees $\tau_{ij}''$, which are branching trees within the fixed environment $(q_{i+1},q_{i+2}, \ldots)$ conditioned on extinction
 before generation $n-i$.
\item All pairs $(R_i,L_i)$  and all subtrees $\tau_{ij}'$, $\tau_{ij}''$  are independent.
\end{itemize}
\end{definition}
The following theorem (see \cite{ge_99,abkv11}) assures that the Geiger and the (pruned) branching tree $T$ conditioned on $\{Z_n>0\}$ have the same distribution: 
\begin{theorem} \label{geigerconstruction} 
For allmost all $\pi$ the conditional distribution of $\langle T\rangle_n$ given  $\Pi=\pi,Z_n>0$ is equal to the distribution of $\langle \mathsf T_{n, \pi}\rangle_n$.
\end{theorem}

\subsection{Branching processes with geometric offspring distributions}
Following the previous section, we can sample a conditional branching tree in the random environment $\pi$ using the Geiger construction. 
Clearly, we require 
\[\rho_{i,n}=\mathbf{P}_\pi(Z_n>0\mid Z_i=1)\ .\]
In the case of geometric (or more generally linear fractional) offspring distributions, 
i.e. $q_i(k)=p_i (1-p_i)^k$, $p_i\in[0,1]$, a direct calculation is feasible. Using the formula for the generating function in
 \cite{kozlov06}[Equation (6) resp. (24)] yields
\[\rho_{i,n}=\Big(\sum_{k=j}^{n} e^{-(S_j-S_i)}\Big)^{-1}\ .\]
Then the distribution in (\ref{geigerdis}) can be written as
\begin{align}
 \mathbf P_\pi& ( R_i=r, L_i=l)= p_i(1-p_i)^{r+l+1} \frac{\rho_{i,n} (1-\rho_{i,n})^{l}}{\rho_{i-1,n}} \ . \label{dis2} 
\end{align}
The offspring distribution in generation $1\leq i\leq n$ in a subtree conditioned on extinction in generation $n$ is given by
\[\tilde q_i(k)=\frac{q_i(k) (1-\rho_{i,n})^k }{1-\rho_{i-1,n}}\ , \ k=1,2,\ldots\ \] 
If $q$ is geometric with parameter $p_i$, then 
\begin{align*}
\tilde q_i(k)& =\frac{p_i (1-p_i)^k (1-\rho_{i,n})^k }{1-\rho_{i-1,n}} = \frac{p_i}{1-\rho_{i-1,n}} \big((1-p_i)(1-\rho_{i,n})\big)^k
\end{align*}
is again geometric with parameter $1-(1-p_i)(1-\rho_{i,n})$. Note that from the definition of $X_i$ and $p_i$,
\[e^{X_i}=\frac{1}{p_i}-1\ .\]

Using the Geiger construction, all offspring numbers of 
indivuals in the conditioned and unconditioned subtrees are independent. As it is well-known, 
the sum of independent geometrically distributed random variables is negative binomially distributed. 

Now, let $Z_i^{(u)}$ be the total number of individuals in the unconditioned trees and $Z_i^{(c)}$ be the total number of individuals in the conditioned 
trees at generation $i$. Thus, following Theorem \ref{geigerconstruction}, we have the following simulation algorithm for 
a branching process in the varying environment $\pi=(q_1,q_2,\ldots)$, 
conditioned on $\{Z_n>0\}$:
\begin{itemize}
\item Calculate $\rho_{i,n}$ for $0\leq i\leq n$.
\item Given $Z_{i-1}^{(u)}$, simulate $Z_i^{(u)}$ as one negative binomially random variable of parameters $(Z_{i-1}^{(u)}, p_{i-1})$.
\item Given $Z_{i-1}^{(c)}$, simulate $Z_i^{(c)}$ as one negative binomially random variable of size parameter $Z_{i-1}^{(c)}$ and success probability   $1-(1-p_i)(1-\rho_{i,n})$.
\item Randomly simulate a pair $(R_i,L_i)$ according to the distribution (\ref{dis2}). Then add $L_i-1$-many individuals to $Z_i^{(c)}$ and $R_i-L_i$-many individuals to $Z^{(u)}_i$.
\item The total number of individuals in generation $i$ is given by $1+Z_i^{(u)}+Z_i^{(c)}$.
\end{itemize}
The simulation amounts to the simulation of one random pair and two independent negative binomially distributed random variables. This allows for very fast simulations.

\section{The functional limit theorem}\label{th}

Now we get down to the announced functional limit theorem. The assumptions are the same as in \cite{abkv11}:

\paragraph{Assumption A1.} {\em Let} $ \mathbf{E} [ X] =0$.

\paragraph{Assumption A2.} {\em The distribution of $X$ 
has  finite variance with 
respect to $\mathbf{P}$ or  (more generally) belongs to the domain of attraction of some zero mean stable law with index $\alpha \in (1,2]$. It is non-lattice.}\\\\
Since $\mathbf{E}[X]=0$ this implies that there is an increasing sequence of positive numbers
\[ a_n  \ = \ n^{1/\alpha} \ell_n \]
with a slowly varying sequence $\ell_1,\ell_2,\ldots$ such that 
\[ \frac {S_n}{a_n} \stackrel{d}{\to} L_1  \]
for $n\rightarrow\infty$, where $L_1$ denotes  a random variable with the above stable distribution. Note that due to the change of measure $X^-$ always has finite variance and 
an infinite variance may only arise from $X^+$. If $\alpha < 2$ this is called the spectrally positive case.

\paragraph{Assumption A3.} {\em For some $\varepsilon > 0$ and some $a \in \mathbb{N}$}
\[ \mathbf{E} [ (\log^+ \zeta(a))^{\alpha + \varepsilon}] \ < \ \infty \ ,\]
where $\log^+ x= \log(x \vee 1)$.\\\\
As is well-known, there is a L\' evy process $L=(L_t)_{0 \le t \le 1}$ including the stable random variable $L_1$ above. 
Let $L^c$ be the corresponding L\'evy process conditioned on having its minimum at time $t=1$. For the precise definition of such a process, we refer to \cite{abkv11}. 
Again define the process $L^r=(L^r_{0\le t \le 1})$, which is the process $L^c$ reflected at its current minimum and given by
\[ L^r_t = L^c_t- \min_{s\le t}L^c_s \ .\]

\begin{theorem}\label{cormain}
Under Assumptions A1 to A3, 
\begin{eqnarray*}
  \big((\tfrac 1{a_n} \log Z_{nt})_{0\leq t\leq 1} \ \big|\ Z_n>0\big) &\stackrel{d}{\rightarrow} &  L^r \ 
\end{eqnarray*}
in the Skorohod space.
\end{theorem}

\noindent
The convergence of the finite-dimensional distributions follows directly from known results: From
\cite{abkv11}[Theorem 1.4] for all $0 \le t \le 1$
\[ \Big(\frac 1{a_n} (\log Z_{nt} - S_{nt}- \min_{s \le t}S_{ns}) \ \Big|\ Z_n>0\Big) \to 0 \]
in probability, and from \cite{abkv11}[Theorem 1.3]
\[ \Big( \frac 1{a_n} S^n \ \Big|\ Z_n>0 \Big) \stackrel d{\to} L^c \ . \]
For tightness, we require the following lemma. Let

\[ \eta_i = \zeta_i(1)= \sum_{y=1}^\infty y^2 Q_i(y) \Big/ m(Q_i)^2 \ . \]

\begin{lemma}\label{le51}
Under Assumptions A1 to A3, for every $\upsilon_n=o(n)$ and $\delta>0$,
\begin{align*}
  e^{-\delta a_n} \sum_{i=0}^{\upsilon_n-1} \eta_{i+1} e^{S_{\upsilon_n}-S_i} \to 0 \ 
\end{align*}
in probability with respect to $\mathbf{P}$.
\end{lemma}
\begin{proof}
By duality
\begin{align*}
 \sum_{i=0}^{\upsilon_n-1} \eta_{i+1} e^{S_{\upsilon_n}-S_i}\stackrel{d}{=}\sum_{i=1}^{\upsilon_n} \eta_{i} e^{S_i}\ .
\end{align*}
Note that $\eta_i\leq a^2 e^{-2X_i}+\zeta_i(a)$, where $a$ is the constant from Assumption A3. Recall $M_n=\max(S_1, \ldots,S_n)$. Thus because of $S_i-2X_i=S_{i-1}-X_i$
\begin{align*}
 \sum_{i=1}^{\upsilon_n} &\eta_{i} e^{S_i} \leq a^2 \sum_{i=1}^{\upsilon_n} e^{S_i-2X_i}+\sum_{i=1}^{\upsilon_n} \zeta_i(a) e^{S_i}\\
&\leq a^2 e^{M_{\upsilon_n}^+} \sum_{i=1}^{\upsilon_n} e^{-X_i}+e^{M_{\upsilon_n}}\sum_{i=1}^{\upsilon_n} \zeta_i(a) \ .
\end{align*}
As $\mathbf{E}[e^{-X}]=\gamma^{-1} \mathbb{E}[e^{X}e^{-X}]=\gamma^{-1}<\infty$, the law of large numbers implies 
\[\sum_{i=1}^{\upsilon_n} e^{-X_i} = O(\upsilon_n) \quad \mathbf{P}\text{-a.s.} \]
Assumption A3 and  the Borel-Cantelli lemma imply
$(\log^+ \zeta_i(a))^{\alpha + \varepsilon}=O(i)$
respectively
\[\zeta_i(a)=O\Big(\exp\big(i^{1/(\alpha+\varepsilon)}\big)\Big) \quad \mathbf{P}\text{-a.s.} \]
Recall that $\upsilon_n=o(n)$ and $a_n=n^{1/\alpha} l(n)$, where $l(n)$ is slowly varying. Thus for every $\delta>0$,
\[\sum_{i=1}^{\upsilon_n} \zeta_i(a) =O\Big(\upsilon_n\exp\big(n^{1/(\alpha+\varepsilon)}\big)\Big) = O\big(e^{\delta a_n/3}\big) \quad \mathbf{P}\text{-a.s.} \]
As a consequence of the functional limit theorem \cite{ka}[Theorem 16.14], $\tfrac{1}{a_n} M_n$  converges in distribution  with respect to $\mathbf{P}$, thus
\[M_{\upsilon_n}=O_P\big(a_{\upsilon_n}\big) =o_P\big(a_n\big) \ .\]
Using these results 
\begin{align*}
 \sum_{i=1}^{\upsilon_n} \eta_{i}e^{S_i} =O_P\big(e^{\delta a_n{/2}}\big) \ .
\end{align*}
This yields the claim.
\end{proof}

The following lemma immediately results.
\begin{lemma}\label{le52}
Under Assumptions A1 to A3, for every $\upsilon_n=o(n)$ and $\delta>0$,
\begin{align*}
  e^{\delta a_n} \mathbf{P}(Z_{\upsilon_n}>0\mid \Pi) \to\infty \ 
\end{align*}
in probability with respect to $\mathbf{P}$.
\end{lemma}
\begin{proof} We use the standard lower estimate for the survival probability (see e.g. \cite{ag} )
\begin{align*}
 \mathbf{P}(Z_{\upsilon_n}>0\mid \Pi)\geq \frac{1}{e^{-S_{\upsilon_n}} + \sum_{k=0}^{\upsilon_n-1} \eta_{k+1} e^{-S_k}}\qquad \text{a.s.}
\end{align*}
Thus it remains to show that 
\begin{align*}
e^{-\delta a_n} \Big(e^{-S_{\upsilon_n}} + \sum_{k=0}^{\upsilon_n-1} \eta_{k+1} e^{-S_k}\Big) \to 0\ 
\end{align*}
which is proved in the same way as Lemma \ref{le51}.
\end{proof}
\begin{proof}[Proof of Theorem \ref{cormain}.] 
It remains to prove tightness.  We use Aldous' criterium  (see e.g. \cite{ka}, p. 314) and show that for any sequence of positive constants $\upsilon_n=o(n)$, any sequence of stopping times $\kappa_1,\kappa_2,\ldots$, with $\kappa_n\le  n$  and any $\delta >0$
\begin{align} \mathbb{P}\big(\tfrac 1{a_n}|\log Z_{(\kappa_n+\upsilon_n)\wedge n} - \log Z_{\kappa_n}|>\delta\mid Z_n>0\big) \to 0 \label{aldous}
\end{align}
for $n \to \infty$. 
First let us fix $0<s<1$ and additionally assume $\kappa_n \le sn$. We show that in
\begin{align}
 \mathbb{P}&(\tfrac 1{a_n}|\log Z_{\kappa_n+\upsilon_n} - \log Z_{\kappa_n}|>\delta\mid Z_n>0) \nonumber\\ 
&= \mathbb P( Z_{\kappa_n+\upsilon_n} > e^{\delta a_n}Z_{\kappa_n} \mid Z_n >0) + \mathbb P( Z_{\kappa_n+\upsilon_n} < e^{-\delta a_n}Z_{\kappa_n} \mid Z_n >0) \label{eq6}
\end{align}
both right-hand terms converge to 0.

Let us treat the first summand in (\ref{eq6}). 
Because of the independence properties of a BPRE, it follows that for stopping times $\kappa_n$
\begin{align}
 \mathbb P( &Z_{\kappa_n+\upsilon_n} > e^{\delta a_n}Z_{\kappa_n} , Z_n >0) \nonumber\\ &= \sum_{z \ge 1, k \le s n}\mathbb P(\kappa_n=k,Z_k=z) 
\mathbb P_z (Z_{\upsilon_n}> e^{\delta a_n} z, Z_{n-k}>0) \label{eq7}
\end{align}
with $\mathbb P_z( \cdot) = \mathbb P(\cdot \mid Z_0=z)$. As to the right-hand probability we distinguish two possibilities: Either one of the $z$ initial individuals has at least one offspring in generation $n-k$ and more than $e^{\delta a_n}$ offspring in generation $\upsilon_n$. Or it has at least one offspring in generation $n-k$ and the other $z-1$ individuals have more than $(z-1)e^{\delta a_n}$ offspring in generation $\upsilon_n$. Thus a.s.
\begin{align}
 \mathbb P_z &(Z_{\upsilon_n}> e^{\delta a_n} z, Z_{n-k}>0\mid \Pi)  \leq z\mathbb P(Z_{\upsilon_n}> e^{\delta a_n}, Z_{n-k}>0\mid \Pi) \nonumber\\
& \quad+  z\mathbb P(Z_{n-k}>0\mid \Pi)\cdot  \mathbb{P}_{z-1}(Z_{\upsilon_n}> (z-1) e^{\delta a_n}\mid \Pi) \ . \label{eq18}
\end{align}
As to the last term by means of Markov's inequality a.s.
\[  \mathbb{P}_{z-1}(Z_{\upsilon_n}> (z-1) e^{\delta a_n}\mid \Pi) \le \frac1{(z-1)e^{\delta a_n}} \mathbb E_{z-1} [Z_{\upsilon_n} \mid \Pi] = e^{-\delta a_n}  e^{S_{\upsilon_n}} \ ,  \] 
thus
\begin{align*}
\mathbb{E} \big[&\mathbb{P}(Z_{n-k}>0\mid\Pi)\cdot  \mathbb{P}_{z-1}(Z_{\upsilon_n}> (z-1) e^{\delta a_n}\mid \Pi)\big] \\
&\leq \mathbb{E}\big[Z_{\upsilon_n}\mathbb{P}(Z_{n-k}>0\mid\Pi, Z_{\upsilon_n}=1)\cdot 1\wedge (e^{-\delta a_n}  e^{S_{\upsilon_n}}) \mid\Pi) \big]\\
&= \mathbb{E}\big[e^{S_{\upsilon_n}}\cdot 1\wedge e^{S_{\upsilon_n}-\delta a_n}\big] \mathbb E[\mathbb{P}(Z_{n-k}>0\mid\Pi, Z_{\upsilon_n}=1)] \\
&= \gamma^{\upsilon_n} \mathbf{E}\big[1\wedge e^{S_{\upsilon_n}-\delta a_n}\big] \mathbb{P}(Z_{n-k-\upsilon_n}>0) \ .
\end{align*}
As $S_{\upsilon_n}-\delta a_n\rightarrow-\infty$ in probability with respect to $\mathbf{P}$ it follows 
by dominated convergence for $n \to \infty$
\[\mathbf{E}\big[\big(1\wedge e^{S_{\upsilon_n}-\delta a_n}\big)\big]\rightarrow 0 \ .\]
Also, applying the remarks after \cite{abkv11}[Theorem 1.1], $\mathbb P(Z_{n}>0)\sim\theta \gamma^{n}/b_n$, where $b_n$ is regularly varying with exponent $1-\alpha^{-1}$. Thus
\begin{align*}
 \gamma^{\upsilon_n} \mathbb{P}(Z_{n-k-\upsilon_n}>0) = O\big(\mathbb{P}(Z_{n-k}>0)\big)
\end{align*}
uniformly in $k \le sn$ and consequently
\begin{align*}
  \mathbb{E}&\big[\mathbb{P}(Z_{n-k}>0\mid \Pi)\cdot  \mathbb{P}_{z-1}(Z_{\upsilon_n}> (z-1) e^{\delta a_n}\mid \Pi)\big] = o( \mathbb{P}(Z_{n-k}>0))
\end{align*}
uniformly in $z\ge 1$ and $k \le sn$.

Next, let us show the same statement for the other term in (\ref{eq18}). For this, we will use \cite{abkv11}[Theorem 4.2]. 
Let $\tilde{\mathsf{T}}$ be the LPP-trest defined in 
\cite{lpp,abkv11} and $\tilde Z_n$ its population size in generation $n$. As above, let
\[ \tau_n = \min \{ k \le n : S_k = \min(S_0, \ldots, S_n)\} \ . \]
Let $m\in\mathbb{N}$ be fixed. Without loss of generality, we write $n$ instead of $n-k\geq (1-s)n$ in 
the following estimates.

In \cite{abkv11}[Equation (4.9)], it is shown that
\[ \mathbf{E}\big[\tilde Z_n\mid\Pi\big]=1+\sum_{i=0}^{n-1} \eta_{i+1} e^{S_{n}-S_i}\ . \]
Using this together with Markov inequality yields
\begin{align*}
 \mathbf{P}&\big(\tilde Z_{\upsilon_n} > e^{\delta a_n} ,\tau_{n-m}=n-m\big) \leq \mathbf{E}\Big[1\wedge \Big(e^{-\delta a_n}\mathbf{E}\big[\tilde Z_n\mid\Pi\big]\Big);\tau_{n-m}=n-m\Big]\\
&\leq \mathbf{E}\Big[1\wedge \Big(e^{-\delta a_n}\mathbf{E}\big[\tilde Z_n\mid\Pi\big]\Big); S_{n-m} < S_{\upsilon_n},S_{\upsilon_n+1}, \ldots, S_{n-m-1}\Big] \\
&= \mathbf{E}\Big[1\wedge \Big(e^{-\delta a_n}\big(1+\sum_{i=0}^{\upsilon_n-1} \eta_{i+1} e^{S_{\upsilon_n}-S_i} \big)\Big)\Big] \mathbf{P}(\tau_{n-m-\upsilon_n}=n-m-\upsilon_n) \ .
\end{align*}
By \cite{abkv11}[Lemma 2.2]  $\mathbf{P}(\tau_n=n)$ is regularly varying, thus, as $n \to \infty$, 
 $\mathbf{P}(\tau_{n-m-\upsilon_n}=n-m-\upsilon_n)\sim \mathbf{P}(\tau_n=n)$. From Lemma 
\ref{le51} and the dominated convergence theorem, it results that
 \[\mathbf{E}\Big[1\wedge \Big(e^{-\delta a_n}\big(1+\sum_{i=0}^{\upsilon_n-1} \eta_{i+1} e^{S_{\upsilon_n}-S_i} \big)\Big)\Big]\rightarrow 0 \ .\] 
 Altogether, as $n \to \infty$, 
 \[ \mathbf P( \tilde Z_{\upsilon_n} > e^{\delta a_n} \mid \tau_{n-m}=n-m) \to 0 \ . \]
Recall that by definition of $\upsilon_n$, $n-\upsilon_n\rightarrow\infty$. Thus using \cite{abkv11}[Theorem 4.2] we get that  uniformly in $k\leq sn$
\[\mathbb{P}(Z_{\upsilon_n}>e^{\delta a_n}, Z_{n-k}>0) = o\big( \mathbb{P}(Z_{n-k}>0)\big)\ .\]
Altogether, from \eqref{eq18}
\[  \mathbb P_z (Z_{\upsilon_n}> e^{\delta a_n} z, Z_{n-k}>0) = o\big( z\mathbb{P}(Z_{n-k}>0)\big) \]
uniformly in $z \ge 1$ and $k \le sn$.
Applying this result to (\ref{eq7}) and changing to the measure $\mathbf{P}$ yields for any $\varepsilon >0$
\begin{align}
 \mathbb P( &Z_{\kappa_n+\upsilon_n} > e^{\delta a_n}Z_{\kappa_n} , Z_n >0)\leq  \varepsilon \sum_{z \ge 1, k \le s n}\mathbb P(\kappa_n=k,Z_k=z) z 
\mathbb P(Z_{n-k}>0)\nonumber \\
&\leq \varepsilon  \sum_{k \le s n}\gamma^k  \mathbf E\big[e^{-S_k} Z_k; \kappa_n=k\big]\mathbb P(Z_{n-k}>0)  \nonumber\\
&= \varepsilon  \sum_{k \le s n} \gamma^k \mathbf E\big[e^{-S_k} \mathbb{E}[Z_k\mid \Pi]; \kappa_n=k\big] \mathbb P(Z_{n-k}>0)\nonumber \\
&= \varepsilon  \sum_{k \le s n}\gamma^k  \mathbf{P}(\kappa_n=k) \mathbb P(Z_{n-k}>0) \ , \label{eq8}
\end{align}
if $n$ is large enough. Finally, applying \cite{abkv11}[Theorem 1.1], $\mathbb P(Z_{n}>0)\sim\theta \gamma^{n}/b_n$ and thus  for large $n$ 
\begin{align*}
 \mathbb P( &Z_{\kappa_n+\upsilon_n} > e^{\delta a_n}Z_{\kappa_n} \mid Z_n >0)\leq 2\varepsilon \frac{b_n}{b_{(1-s)n}} \ . 
\end{align*}
Taking the limit $n\rightarrow\infty$ and then $\varepsilon\rightarrow 0$
\[ \mathbb P( Z_{\kappa_n+\upsilon_n} > e^{\delta a_n}Z_{\kappa_n} \mid Z_n >0) \to 0\ . \]

Next, let us turn to the second summand in (\ref{eq6}). Applying similar arguments as before, 
\begin{align*}
 \mathbb P_z&(Z_{\upsilon_n} <z e^{-\delta a_n}, Z_{n-k}>0\mid\Pi)\leq z\mathbb{P}(Z_{\upsilon_n} <e^{-\delta a_n}, Z_{n-k}>0\mid\Pi)\\
 &+ z\mathbb{P}(Z_{n-k}>0\mid\Pi) \mathbb{P}_{z-1}\big(Z_{\upsilon_n}<(z-1) e^{-\delta a_n}\mid\Pi\big)\  \text{a.s.}
\end{align*}
As $e^{-\delta a_n}<1$, the first right-hand term vanishes, also for $ze^{-\delta a_n}<1$ the left-hand side  is $0$.  Thus the inequality becomes
\begin{align*}
 \mathbb P_z&(Z_{\upsilon_n} <z e^{-\delta a_n}, Z_{n-k}>0\mid\Pi)\\
&\leq z\mathbf{1}_{\{z\geq e^{\delta a_n}\}}\mathbb{P}(Z_{n-k}>0\mid\Pi) \mathbb{P}_{z-1}\big(Z_{\upsilon_n}<(z-1) e^{-\delta a_n}\mid\Pi\big)\  \text{a.s.}
\end{align*}
Conditioned on the environment, all $(z-1)$-many subtrees of the branching process are independent. Of these  subtrees, 
at most  $(z-1)e^{-\delta a_n} $-many may survive until generation $\upsilon_n$.  Thus, letting $Y$ 
be a binomially distributed random variable with parameters $(z-1,p)$, $p=\mathbb{P}(Z_{\upsilon_n}>0\mid\Pi)$, by means of Chebyshev's inequality for $z>1$ 
\begin{align*}
\mathbb{P}_{z-1}&\big(Z_{\upsilon_n}<(z-1) e^{-\delta a_n}\mid\Pi\big) \mathbf{1}_{\{z\geq e^{\delta a_n}\}} \le \mathbb{P}\big(Y<(z-1) e^{-\delta a_n}\mid\Pi\big) \mathbf{1}_{\{z\geq e^{\delta a_n}\}}\\
&\le \mathbf{1}_{\{e^{-\delta a_n}>p/2\}} + \mathbb{P}\big(Y<(z-1) p/2\mid\Pi\big) \mathbf{1}_{\{z\geq e^{\delta a_n}\}}\\
&\le \mathbf{1}_{\{e^{\delta a_n}p<2\}} + \frac 4{(z-1)p} \mathbf{1}_{\{z\geq e^{\delta a_n}\}}\\
&\le  \frac 2{e^{\delta a_n}p}+ \frac 4{(e^{\delta a_n}-1)p} \le \frac 6{(e^{\delta a_n}-1)p} \ . 
\end{align*}
It follows
\begin{align*}
\mathbb P_z(&Z_{\upsilon_n} < z e^{-\delta a_n}, Z_{n-k}>0) \le 6z \mathbb E\big[ \mathbb P(Z_{n-k}>0 \mid \Pi) \big(1\wedge ((e^{\delta a_n}-1)p)^{-1}\big)\big]\\
&\le 6z \mathbb E\big[ Z_{\upsilon_n} \mathbb P(Z_{n-k-\upsilon_n}>0 \mid \Pi, Z_{\upsilon_n} =1) \big(1\wedge ((e^{\delta a_n}-1)p)^{-1}\big)\big] \\
&= 6z \mathbb E\big[e^{S_{\upsilon_n}} \big(1\wedge ((e^{\delta a_n}-1)p)^{-1}\big)\big] \mathbb P(Z_{n-k-\upsilon_n}>0)\\
&= 6z \gamma^{\upsilon_n}\mathbf E\big[\big(1\wedge ((e^{\delta a_n}-1)p)^{-1}\big)\big] \mathbb P(Z_{n-k-\upsilon_n}>0)\ .
\end{align*}
From Lemma \ref{le52} and dominated convergence, we obtain
\[ \mathbf E\big[\big(1\wedge ((e^{\delta a_n}-1)p)^{-1}\big)\big] \to 0 \]
for $n \to \infty$, thus as above
\[  \mathbb P_z (Z_{\upsilon_n}<e^{-\delta a_n} z, Z_{n-k}>0) = o\big( z\mathbb{P}(Z_{n-k}>0)\big) \]
uniformly in $z \ge 1$ and $k \le sn$, and in much the same way as above we may conclude
\[ \mathbb P( Z_{\kappa_n+\upsilon_n} < e^{-\delta a_n}Z_{\kappa_n} \mid Z_n >0) \to 0 \]
as $n \to \infty$. From \eqref{eq6} we see that \eqref{aldous} is satisfied for any sequence of stopping times $\kappa_n$ such that $\kappa_n \le sn$ with some $s<1$.

Finally, we show that \eqref{aldous} also holds for all stopping times $\kappa_n\leq n$. Let $0\leq s<1$. Then we get that for large $n$ 
\begin{align*}
 \mathbb{P}&\big(\tfrac 1{a_n}\big|\log Z_{(\kappa_n+\upsilon_n)\wedge n}-\log Z_{\kappa_n}\big|>\delta\mid Z_n>0\big)  \\
&\leq \mathbb{P}\big(\tfrac 1{a_n}\big|\log Z_{\kappa_n\wedge sn+\upsilon_n}-\log Z_{\kappa_n\wedge sn}\big|>\delta\mid Z_n>0\big)\\
&\quad + 
2\mathbb{P}\Big(\tfrac 1{a_n}\sup_{s\leq t\leq 1} \big|\log Z_{nt}-\log Z_{n}\big|>\delta\mid Z_n>0\Big)\ . 
\end{align*}
Since $\kappa_n \wedge sn$ is again a stopping time, taking the $\limsup$, the first term above vanishes for every $0\leq s<1$. Thus it is enough to show that the second term can be made arbitrarily small in the limit, 
if we choose $s$ close enough to 1. Now
\begin{align*}
 \mathbf{1}_{\{Z_n>0\}}&\tfrac 1{a_n}\sup_{s\leq t\leq 1} \big|\log Z_{nt}-\log Z_{n}\big|\\
 &\le \mathbf{1}_{\{Z_n>0\}}\tfrac 1{a_n}\sup_{s\leq t\leq 1} \big(\log Z_{nt}+\log Z_{n}\big)\\
 &\leq  \tfrac 1{a_n}\sup_{0\leq t\leq 1} \big(\log^+ Z_{nt}- (S_{nt}-S_n)\big)+ \tfrac 1{a_n} \sup_{s\leq t\leq 1}| S_{nt}-S_n|+\tfrac 1{a_n} \log^+ Z_n \ .
\end{align*}
Conditioned on $Z_n>0$, $Z_n$ has a limiting distribution on $\mathbb{N}$, see \cite{abkv11}[Theorem 1.2]. Thus, as $n\rightarrow\infty$, for $\delta >0$
\[\lim_{n\rightarrow\infty}\mathbb{P}\big(\tfrac 1{a_n} \log^+ Z_n>\delta\mid Z_n>0\big)=0\ . \]
As to the second term, using \cite{abkv11}[Theorem 1.3] and letting $n\rightarrow \infty$, for every $\varepsilon>0$, 
\begin{align*}
  \mathbb{P}\big(\tfrac 1{a_n} \sup_{s \leq t\leq 1} |S_{n}-S_{nt}|> \delta \mid Z_n>0\big)<\varepsilon
\end{align*}
if only $s$ is chosen close enough to 1. 

Finally, as conditioned on $\Pi$, $Z_{nt}e^{-S_{nt}}$ is a non-negative martingale with mean 1, we use the Doob inequality to get that
\begin{align*}
 \mathbb{P}&\big(\tfrac 1{a_n}\sup_{0\leq t\leq 1} \big(\log Z_{nt}- (S_{nt}-S_n)\big)>\delta, Z_n>0\big)\\
&\leq \mathbb{P}\big(\sup_{0\leq t\leq 1} Z_{nt}e^{-S_{nt}+S_n}>e^{\delta a_n}\big)\\
&= \mathbb{E}\Big[ \mathbb{P}\big(\sup_{0\leq t\leq 1} Z_{nt}e^{-S_{nt}}>e^{\delta a_n-S_n}\mid \Pi\big)\Big]\\
&\leq \mathbb{E}\big[e^{S_n-\delta a_n} \big] = \gamma^n e^{-\delta a_n} \ .
\end{align*}
Again by \cite{abkv11}[Theorem 1.1], $\gamma^n e^{-\delta a_n}=o\big(\mathbb{P}(Z_n>0)\big)$. Altogether, this yields 
\[ \limsup_{n\rightarrow\infty} \mathbb{P}\big(\tfrac 1{a_n}\sup_{s\leq t\leq 1} \big|\log Z_{nt}-\log Z_{n}\big|>3\delta\mid Z_n>0\big)\leq \varepsilon \]
if only $s$ is close enough to 1. Thus we have proved \eqref{aldous} for every sequence of stopping times $\kappa_n\leq n$ which yields tightness.
\end{proof}

\end{document}